\documentclass[11pt,reqno]{amsart}

\usepackage{amssymb}
\usepackage{amscd}
\usepackage{amsfonts}
\usepackage{mathrsfs}
\usepackage{setspace}
\usepackage{version}
\usepackage{mathtools}
\usepackage{enumerate}
\usepackage[english]{babel}
\usepackage{xcolor}
\usepackage[colorlinks, citecolor=blue, linkcolor=red]{hyperref}

\newtheorem{theorem}{Theorem}[section]
\newtheorem{lemma}[theorem]{Lemma}
\newtheorem{proposition}[theorem]{Proposition}
\newtheorem{corollary}[theorem]{Corollary}
\theoremstyle{definition}
\newtheorem{definition}[theorem]{Definition}

\newtheorem{example}[theorem]{Example}

\theoremstyle{remark}
\newtheorem{remark}[theorem]{Remark}
\numberwithin{equation}{section}

\newcommand{\BB}{\mathcal{B}}

\newcommand{\FF}{\mathcal{F}}

\renewcommand{\SS}{\mathscr{S}}

\newcommand{\field}[1]{\mathbb{#1}}

\newcommand{\R}{\field{R}}
\newcommand{\N}{\field{N}}

\newcommand{\E}{\field{E}}

\renewcommand{\P}{\field{P}}

\newcommand{\id}{\mathop{\text{\upshape{id}}}}

\newcommand{\supp}{\mathop{\rm{supp}}}

\newcommand{\de}{\delta}
\newcommand{\ep}{\varepsilon}

\newcommand{\la}{\lambda}

\newcommand{\si}{\sigma}

\newcommand{\ph}{\varphi}

\newcommand{\La}{\Lambda}
\newcommand{\Si}{\Sigma}

\newcommand{\Om}{\Omega}

\renewcommand{\SS}{\mathcal{S}}

\newcommand{\BUC}{\text{\rm{BUC}}}

\begin{document}

	\title[Convex semigroups on $L^p$-like spaces]{Convex semigroups on $L^p$-like spaces}
	\author{Robert Denk}
	\address{Department of Mathematics and Statistics, University of Konstanz, Germany}
	\email{robert.denk@uni-konstanz.de}
	
	\author{Michael Kupper}
	\address{Department of Mathematics and Statistics, University of Konstanz, Germany}
	\email{kupper@uni-konstanz.de}
	
	\author{Max Nendel}
	\address{Center for Mathematical Economics, Bielefeld University, Germany}
	\email{max.nendel@uni-bielefeld.de}

	\date{\today}

	\thanks{Financial support through the German Research Foundation via CRC 1283 is gratefully acknowledged. We thank Daniel Bartl, Jonas Blessing, Liming Yin and Jos\'e Miguel Zapata Garc\'ia for helpful discussions and comments. We are also indebted to two anonymous referees for their helpful suggestions leading to a major improvement of the manuscript.}
	
	\subjclass[2010]{}
	
	\begin{abstract}
	In this paper, we investigate convex semigroups on Banach lattices with order continuous norm, having $L^p$-spaces in mind as a typical application. We show that the basic results from linear $C_0$-semigroup theory extend to the convex case. We prove that the generator of a convex $C_0$-semigroup is closed and uniquely determines the semigroup whenever the domain is dense. Moreover, the domain of the generator is invariant under the semigroup; a result that leads to the well-posedness of the related Cauchy problem. In a last step, we provide conditions for the existence and strong continuity of semigroup envelopes for families of $C_0$-semigroups. The results are discussed in several examples such as semilinear heat equations and nonlinear integro-differential equations.

       \smallskip
       \noindent \emph{Key words:} Convex semigroup, nonlinear Cauchy problem, well-posedness and uniqueness, Hamilton-Jacobi-Bellman equation

       \smallskip
	\noindent \emph{AMS 2010 Subject Classification:} 47H20; 35A02; 35A09
	\end{abstract}

	\maketitle
	
	\setcounter{tocdepth}{1}

\section{Introduction}

Decision-making in a dynamic random environment naturally leads to so-called stochastic optimal control problems.\ These type of problems arise in numerous applications in economics and mathematical finance, cf.~Fleming-Soner \cite{MR2179357} or Pham \cite{MR2533355}. Examples include irreversible investments, endogenous growth models, such as the AK-model, portfolio optimization, as well as superhedging and superreplication under model uncertainty. In this context, the dynamic programming principle typically leads to convex partial differential equations, so-called Hamilton-Jacobi-Bellman (HJB) equations, where, intuitively speaking, the convexity comes from optimizing among a certain class of Markov processes, each one linked to a linear PDE via its infinitesimal generator. One classical approach to treat nonlinear partial differential equations uses the theory of maximal monotone or m-accretive operators; see, e.g., Barbu \cite{Barbu10}, B\'enilan-Crandall \cite{Benilan-Crandall91}, Br\'ezis \cite{Brezis71}, Evans \cite{Evans87}, Kato \cite{Kato67}, and the references therein. To show that an accretive operator is m-accretive, one has to prove that $1+h\mathcal A$ is surjective for small $h>0$. However, in many cases it is hard to verify this condition. This was one of the reasons for the introduction of viscosity solutions, where existence and uniqueness holds due to the milestone papers by Crandall-Ishii-Lions \cite{Crandall-Ishii-Lions92},\cite{Crandall-Lions83} and Ishii \cite{Ishii89}. Although Perron's method (providing the existence) and Ishii's lemma (in order to obtain uniqueness) are applicable to a large class of HJB equations, solvability in terms of viscosity solutions is a rather weak notion of well-posedness, in the sense that viscosity solutions without any further results on regularity are a priori not differentiable in any sense.

Inspired by Nisio \cite{MR0451420}, in this paper, we approach convex differential equations with a semigroup approach. We extend classical results from semigroup theory regarding uniqueness of the semigroup in terms of the generator, space and time regularity of solutions in terms of initial data, more precisely, invariance of the domain under the semigroup, and classical well-posedness of related Cauchy problems to the convex case.

Given a $C_0$-semigroup $S=(S(t))_{t\in [0,\infty)}$ of linear operators on a Banach space $X$ with generator $A\colon D(A)\subset X\to X$, it is well known that the domain $D(A)$ is invariant under $S$, that is, $S(t)x\in D(A)$ for all $x\in D(A)$ and $t\ge 0$. Moreover, it holds
\begin{equation}\label{sgandgen}
AS(t)x=S(t)Ax\quad \text{for all }x\in D(A)\mbox{ and }t\ge 0.
\end{equation}
This relation is fundamental in order to prove the invariance of the domain under the semigroup, that the semigroup $S$ is uniquely determined through its generator, and results in the classical well-posedness of the associated Cauchy problem.

In this work, we show that the aforementioned fundamental results from linear semigroup theory extend to the convex case, if the underlying space $X$ satisfies some additional properties, and the right-hand side of \eqref{sgandgen} is replaced by a directional derivative. To that end, we assume that $X$ is Dedekind $\si$-complete and that $\lim_{n\to \infty}\|x_n- \inf_{m\in \N}x_m\|=0$ for every decreasing sequence $(x_n)_{n\in\mathbb{N}}$ in $X$ which is bounded below, having $X=L^p(\mu)$ for $p\in [1,\infty)$ and an arbitrary measure $\mu$ in mind as a typical example. Then, a convex $C_0$-semigroup $S$ on $X$ is a family $(S(t))_{t\in[0,\infty)}$ of bounded convex operators $X\to X$, such that, for every $x\in X$, it holds $S(0)x=x$,  $S(t+s)x=S(t)S(s)x$ for all $s,t\ge 0$, and $S(t)x\to x$ as $t\downarrow 0$.
Defining its generator $A\colon D(A)\subset X\to X$ as in the linear case by
\[
Ax:=\lim_{h\downarrow 0} \frac{S(h)x-x}{h},\quad\mbox{
where}\quad D(A):=\Big\{x\in X : \lim_{h\downarrow 0}\frac{S(h)x-x}{h}\text{ exists}\Big\},\] we show that the convex $C_0$-semigroup $S$ leaves the domain $D(A)$ invariant. Moreover, the map $[0,\infty)\to X$, $t\mapsto S(t)x$ is continuously differentiable for all $x\in D(A)$, and the time derivative is given by
\[
AS(t)x=S'(t,x)Ax:=\inf_{h>0}\frac{S(t)(x+h Ax)-S(t)x}{h}.
\]
The right-hand side of this equation is the directional derivative or G\^{a}teaux derivative of the convex operator $S(t)$ at $x$ in direction $Ax$. In particular, if $S(t)$ is linear, the G\^{a}teaux derivative simplifies to $S'(t,x)Ax=S(t)Ax$, which is consistent with \eqref{sgandgen}. We further show that the generator $A$ is always a closed operator which uniquely determines the semigroup $S$ on dense subsets of the domain $D(A)$. As a consequence, $y(t):=S(t)x$, for $x\in D(A)$, defines the unique classical solution to the abstract Cauchy problem
\[
{\rm (CP)}\qquad \begin{cases}y'(t)=Ay(t),& \text{for all }t\geq 0,\\
\,y(0)=x.
\end{cases}
\]
Motivated by stochastic optimal control problems, we then specialize on a setup, where, for an arbitrary index set $\La$, we consider families $(S_\la)_{\la\in \La}$ of convex monotone semigroups $S_\lambda = (S_\lambda(t))_{t\ge 0}$. We then address the question of the existence of a smallest upper bound $S$ of the family $(S_\la)_{\la\in \La}$ within the class of semigroups. We provide conditions that ensure the existence and strong continuity of the smallest upper bound $S$, which we refer to as the (upper) semigroup envelope, making the above mentioned results on convex semigroups applicable to this setting. Formally, the generator $A$ of the semigroup envelope $S$ corresponds of the operator $\sup_{\la\in \La}A_\la$, defined on $\bigcap_{\lambda\in\Lambda} D(A_\lambda)$, where, for $\la\in \La$, $A_\la$ is the generator of $S_\la$. In this case, at least formally, the Cauchy problem (CP) results in an abstract Hamilton-Jacobi-Bellman-type equation of the form
\begin{equation}\label{nonlinPDE}
 \partial_t u(t)=\sup_{\la\in \La} A_\la u(t)\quad \text{for }t\geq 0, \quad u(0)=u_0.
\end{equation}
Following Nisio \cite{MR0451420},  Denk-Kupper-Nendel \cite{dkn2} and Nendel-R\"ockner \cite{roecknen}, where the existence of a semigroup envelope, under certain conditions, has been shown for families of semigroups on $\BUC$, we provide conditions for convolution semigroups on $L^p(\mu)$ that make the aforementioned relation rigorous. In general, the obtained domain $D(A)$ will be larger than the natural domain $\bigcap_{\lambda\in\Lambda} D(A_\lambda)$. However, our results imply the existence and classical differentiability of the solution even for initial values in $D(A)$. We remark that for generators of L\'evy processes in $\BUC$ under uncertainty, recent results were obtained, e.g., in Denk-Kupper-Nendel \cite{dkn2},  Hollender \cite{Hollender16}, K\"uhn \cite{Kuehn19}, Nendel-R\"ockner \cite{roecknen},  and Neufeld-Nutz \cite{NutzNeuf}. Fully nonlinear equations in the strong $L^p$-setting were recently considered, e.g., by Krylov \cite{Krylov17},\cite{Krylov18},\cite{Krylov18a}.

The structure of the paper is as follows. In Section \ref{sec:convexsemigroup}, we introduce the setting and state basic results on convex $C_0$-semigroups, which can be derived from a uniform boundedness principle for convex operators.
Section \ref{sec:dedcomplete} includes the main results on convex $C_0$-semigroups, their generators and related Cauchy problems. In particular, we provide invariance of the domain, uniqueness of the semigroup in terms of the generator and classical well-posedness of the related Cauchy problem. In Section \ref{sec.envelopes}, we consider the smallest upper bound, called the (upper) semigroup envelope, of a family of convex monotone semigroups. In Section \ref{sec:convsglp}, we provide conditions for the existence and strong continuity of the semigroup envelope for families $(S_\la)_{\la\in \La}$ of linear convolution semigroups on $L^p(\mu)$, and relate the generator of the semigroup envelope to $\sup_{\la\in \La} A_\la$, i.e., the smallest upper bound of the generators $(A_\la)_{\la\in \La}$ of $(S_\la)_{\la\in \La}$. We illustrate the results with Example \ref{ex:smallg} and Example \ref{ex.compoundpois}. In the appendix, we collect some additional results on bounded convex operators on general Banach lattices including a version of the uniform boundedness principle for convex operators.

\section{Notation and preliminary results}\label{sec:convexsemigroup}
In this section, we introduce our setup, define convex semigroups, which are the central object of this manuscript, and discuss some technical properties of these semigroups, which will be fundamental for the analysis in the subsequent sections.

Let $X$ be a Banach lattice. For an operator $S\colon X\to X$, we define
\[
 \|S\|_r:=\sup_{x\in B(0,r)}\|S x\|
\]
for all $r>0$, where $B(x_0,r):=\{x\in X\colon \|x-x_0\|\leq r\}$ for $x_0\in X$. We say that an operator $S\colon X\to X$ is \textit{convex} if $S \big(\lambda x+(1-\lambda) y\big)\leq \lambda S x+ (1-\lambda) Sy$ for all $\lambda \in[0,1]$, \textit{positive homogeneous} if $S (\lambda x)= \lambda S x$ for all $\lambda>0$, \textit{sublinear} if $S$ is convex and positive homogeneous, \textit{monotone} if $x\le y$ implies $Sx\le Sy$ for all $x,y\in X$, and \textit{bounded} if $\|S\|_r<\infty$ for all $r>0$.
\begin{definition}\label{defstochsemi}
 A family $S=(S(t))_{t\in[0,\infty)}$ of bounded operators $X\to X$ is called a \textit{semigroup} on $X$ if
 \begin{enumerate}
  \item[(S1)] $S(0)x=x$ for all $x\in X$,
  \item[(S2)] $S(t+s)x=S(t)S(s)x$ for all $x\in X$ and $s,t\in[0,\infty)$.
 \end{enumerate}
 In this case, we say that $S$ is a \textit{$C_0$-semigroup} if, additionally,
 \begin{enumerate}
  \item[(S3)] $S(t)x\to x$ as $t\downarrow 0$ for all $x\in X$.
 \end{enumerate}
We say that  $S$ is \textit{convex}, \textit{sublinear} or \textit{monotone} if $S(t)$ is convex, sublinear or monotone for all $t\ge 0$, respectively.
\end{definition}

Throughout the rest of this section, let $S$ be a convex $C_0$-semigroup on $X$. For $t\ge 0$ and $x\in X$, we define the convex operator
$S_x(t)\colon X\to X$ by
\[
S_x(t) y:=S(t)(x+y)-S(t)x.
\]

\begin{proposition}\label{sem:unibound}
 Let $T>0$ and $x_0\in X$. Then, there exist $L\geq 0$ and $r>0$ such that
 \[
  \sup_{t\in [0,T]}\|S_x(t)y\|\leq L\|y\|
 \]
 for all $x\in B(x_0,r)$ and $y\in B(0,r)$.
\end{proposition}

\begin{proof}
 It suffices to show that
 \begin{equation}\label{eq:help1}
  \sup_{0\leq t\leq T}\|S(t)x\|<\infty
 \end{equation}
 for all $x\in X$. Indeed, under \eqref{eq:help1}, it follows from Theorem \ref{lem:unibound} b) that there exists some $r>0$ such that
 $b:=\sup_{x\in B(x_0,r)}\sup_{0\le t\le T}\|S_x(t)\|_r<\infty$.
 Since $S_x(t)$ is convex and  $S_x(t)0=0$,
 we obtain from Lemma \ref{lem:lip} that
  \[
 \|S_x(t) y\|\leq  \tfrac{2b}{r} \|y\|
 \]
 for all $t\in[0,T]$, $x\in B(x_0,r)$ and $y\in B(0,r)$.

 In order to prove \eqref{eq:help1}, let $x\in X$. Since $S(t)x\to x$ as $t\downarrow 0$, there exists some $n\in \N$ such that
 \[ R:=\sup_{h\in [0,\delta)}\|S(h)x\|<\infty,\]
 where $\delta:=\tfrac{T}{n}$. Since $S(t)$ is bounded for all $t\geq 0$, it holds
 \[
  c:=\max_{0\leq k\leq n} \|S(k\delta)\|_R<\infty.
 \]
 Now, let $t\in [0,T]$. Then, there exist $k\in \{0,\ldots,n\}$ and $h\in [0,\delta)$ such that $t=k\delta+h$. Since $\|S(h)x\|\leq R$, it follows that
 $\|S(t)x\|=\|S(k\delta)S(h)x\|\leq c$. This proves \eqref{eq:help1} and thus completes the proof.
\end{proof}

\begin{remark}
 If $S$ is sublinear, then there exist $\omega \in \R$ and $M\geq 1$ such that
 \begin{equation}\label{eq:growth}
 \|S(t)x\|\leq Me^{\omega t}\|x\|
 \end{equation}
 for all $x\in X$ and $t\in[0,\infty)$. Indeed,  by Proposition \ref{sem:unibound} and sublinearity of the semigroup $S$, one has $$M:=\sup_{t\in [0,1]} \sup_{x\in X}\frac{\|S(t)x\|}{\|x\|}<\infty.$$
 Set $\omega:=\log M$. Then, for all $t\in [0,\infty)$, there exists some $m\in \N$ with $t<m\leq t+1$. By the semigroup property, it follows that
 \[
 \|S(t)x\|=\big\|S\big(\tfrac{t}{m}\big)^mx\big\|\leq M^m\|x\| \leq M^{t+1}\|x\|=Me^{\omega t}\|x\|
 \]
 for all $x\in X$. 
\end{remark}

\begin{corollary}\label{loclipschitz}
 Let $T>0$ and $x_0\in X$. Then, there exist $L\geq 0$ and $r>0$ such that
 \[
  \sup_{t\in [0,T]}\|S(t)y-S(t)z\|\leq L\|y-z\|
 \]
 for all $y,z\in B(x_0,r)$.
\end{corollary}

\begin{proof}
 By Proposition \ref{sem:unibound}, there exist $L\geq 0$ and $r>0$ such that
  \[
  \sup_{t\in [0,T]}\|S_x(t)y\|\leq L\|y\|
 \]
 for all $x\in B(x_0,2r)$ and $y\in B(0,2r)$. Now, let $y,z\in B(x_0,r)$. Then, $y-z\in B(0,2r)$, and we thus obtain that
 \[
  \sup_{t\in [0,T]}\|S(t)y-S(t)z\|=\sup_{t\in [0,T]}\|S_z(t)(y-z)\| \leq L\|y-z\|,
 \]
 which shows the desired Lipschitz continuity.
\end{proof}

\begin{corollary}\label{cor:cont}
 The map $[0,\infty)\to X,\; t\mapsto S(t)x$ is continuous for all $x\in X$.
\end{corollary}

\begin{proof}
 Let $t\geq 0$ and $x\in X$. Then, by Corollary \ref{loclipschitz}, there exist $L\geq 0$ and $r>0$ such that
 \[
  \sup_{s\in [0,t+1]}\|S(s)y-S(s)x\|\leq L\|y-x\|
 \]
 for all $y\in B(x,r)$. Moreover, there exists some $\de\in (0,1]$ such that $\|S(h)x-x\|\leq r$ for all $h\in [0,\de]$. For $s\geq 0$ with $|s-t|\leq \de$, it follows that
 \[
  \|S(t)x-S(s)x\|=\|S(s\wedge t)S(|t-s|)x-S(s\wedge t)x\|\leq L\|S(|t-s|)x-x\|\to 0
 \]
 as $s\to t$.
\end{proof}

\begin{corollary}\label{cor:crossconv}
Let $(x_n)_{n\in\mathbb{N}}$ and $(y_n)_{n\in\mathbb{N}}$ be two sequences in $X$ with $x_n\to x\in X$ and $y_n\to y\in X$, and $(h_n)_{n\in\mathbb{N}}$ be a sequence in $(0,\infty)$ with $h_n\downarrow 0$. Then, $S_{x_n}(h_n)y_n\to y$.
\end{corollary}

\begin{proof}
 We first show that $S(h_n)x_n\to x$. By Corollary \ref{loclipschitz}, there exist $L\geq 0$ and $r>0$ such that
 \[
  \sup_{t\in [0,1]}\|S(t)z-S(t)x\|\leq L\|z-x\|.
 \]
 for all $z\in B(x,r)$. Hence, for $n\in \N$ sufficiently large, we obtain that
 \begin{align*}
  \|S(h_n)x_n- x\|&\leq \|S(h_n)x_n-S(h_n)x\|+\|S(h_n)x-x\|\\
  &\leq L\|x_n-x\|+\|S(h_n)x-x\|.
 \end{align*}
 This shows that $S(h_n)x_n\to x$ as $n\to \infty$. As a consequence,$$S_{y_n}(h_n)x_n=S(h_n)(x_n+y_n)-S(h_n)y_n\to (x+y)-y=x$$
 as $n\to\infty$. The proof is complete.
\end{proof}

\begin{proposition}\label{domainlip}
Let $x\in X$ with
\[
\sup_{h\in (0,h_0]}\bigg\|\frac{S(h)x-x}{h}\bigg\|<\infty\quad  \text{for some }h_0>0.
\]
Then, the map $[0,\infty)\to X$, $t\mapsto S(t)x$ is locally Lipschitz continuous, i.e., for every $T>0$, there exists some $L_T\geq 0$ such that $\|S(t)x-S(s)x\|\leq L_T|t-s|$ for all $s,t\in [0,T]$.
\end{proposition}

\begin{proof}
 Since the map $[0,\infty)\to X,\; t\mapsto S(t)x$ is continuous by Corollary \ref{cor:cont}, there exists some constant $C_T\geq 0$ such that
 \[
  \sup_{t\in (0,T]}\frac{\|S(t)x-x\|}{t}\leq C_T.
 \]
 By Corollary \ref{loclipschitz}, there exist $L\geq 0$ and $r>0$ such that
 \[
  \sup_{t\in [0,T]}\|S(t)y-S(t)z\|\leq L\|y-z\| \quad \text{for all }y,z\in B(x,r).
 \]
 Further, there exists some $n\in \N$ such that $\sup_{h\in [0,\delta]}\|S(h)x-x\|\leq r$, where $\delta:=\tfrac{T}{n}$. Now, let $L_T:=LC_T$ and $s,t\in [0,T]$ with $s\leq t$. If $t-s\in [0,\delta]$, we have that
 \[
  \|S(t)x-S(s)x\|\leq L\|S(t-s)x-x\|\leq L_T (t-s).
 \]
 In general, there exist $k\in \{0,\ldots, n-1\}$ and  $h\in [0,\delta]$ such that $t-s=k\delta +h$. Then,
 \begin{align*}
  \|S(t)x-S(s)x\|&\leq \|S(t)x-S(s+k\delta)x\|+\sum_{j=1}^k \big\|S(s+j\delta)x-S\big(s+(j-1)\delta\big)x\big\|\\
  &\leq L_T \big(t-(s+k\de)\big)+L_T k\delta=L_T(t-s).
 \end{align*}
 The proof is complete.
\end{proof}

\section{Generators of convex semigroups and related Cauchy problems}\label{sec:dedcomplete}
In this section, we assume that $X$ is a Banach lattice with order continuous norm, i.e., for every net $(x_\alpha)_\alpha$ with $x_\alpha \downarrow 0$, we have $\|x_\alpha\|\to 0$. Note that this is equivalent to the condition that $X$ is \textit{Dedekind $\si$-complete}, i.e.,\ any countable non-empty subset of $X$, which is bounded above, has a supremum, with \textit{$\sigma$-order continuous} norm $\|\cdot \|$, i.e., for every decreasing sequence $(x_n)_{n\in\mathbb{N}}$ with $\inf_{n\in \N}x_n=0$, it holds $\lim_{n\to\infty}\|x_n\|=0$ (see \cite[Theorem 2.4.2]{MR1128093} or \cite[Theorem 1.1]{wnuk1999banach}). Recall that order continuity of the norm $\|\cdot\|$ also implies the \textit{Dedekind super completeness} of $X$, i.e., every non-empty subset which is bounded above has a countable subset with identical supremum, see, for instance, \cite[Corollary 1 to
Theorem II.5.10]{MR0423039} or \cite[Theorem 1.1]{wnuk1999banach}. Moreover, we would like to point out that separability together with Dedekind $\si$-completeness of $X$ implies order continuity of the norm, cf.\ \cite[Exercise 2.4.1]{MR1128093} or \cite[Corollary to Theorem II.5.14]{MR0423039}.\ Typical examples for $X$ are given by $X=L^p(\mu)$ for $p\in [1,\infty)$ and some measure $\mu$, Orlicz spaces, and the space $X=c_0$ of all sequences converging to $0$. For more details on these spaces, we refer to \cite[Section 2.4]{MR1128093} or \cite{wnuk1999banach}. Again, let $S$ be a convex $C_0$-semigroup on $X$.

\begin{definition}
 We define the \textit{generator} $A\colon D(A)\subset X\to X$ of $S$ by
 \begin{align*}\label{def:dom}
  A x:= \lim_{h\downarrow 0} \tfrac{S(h)x-x}{h},\quad\mbox{where}\quad D(A):=\bigg\{x\in X\colon \frac{S(h)x-x}{h}\text{ is convergent for } h\downarrow 0\bigg\}.
 \end{align*}
\end{definition}
In this subsection, we investigate properties of the generator $A$ and its domain $D(A)$. A fundamental ingredient for the analysis is the directional derivative of a convex operator.
Fix $t\ge 0$. Since $S(t)\colon X\to X$ is a convex operator, the function
\[
\R\setminus\{0\}\to X, \quad h\mapsto\frac{S(t)(x+h y)-S(t)x}{h}
\]
is increasing for all $x,y\in X$. In particular,
\[
 -S_x(t)(-y)\leq \frac{S(t)(x-h y)-S(t) x}{-h}\leq \frac{S(t)(x+h y)-S(t) x}{h}\leq S_x(t)y
\]
for $x,y\in X$ and $h\in (0,1]$. Since for all $x,y\in X$ and every sequence $(h_n)_{n\in\mathbb{N}}$ in $(0,\infty)$ with $h_n\to  0$, it holds
\[
\inf_{n\in \N} \frac{S(t)(x+h_n y)-S(t)x}{h_n}\in X \quad\text{and}\quad \sup_{n\in \N} \frac{S(t)x-S(t)(x-h_n y)}{h_n}\in X,
\]
the operators
\begin{equation}\label{def:Sprime}
 S^\prime_+(t,x)y:=\inf_{h> 0} \frac{S(t)(x+h y)-S(t)x}{h}\quad\mbox{and}\quad
 S^\prime_-(t,x)y:=\sup_{h< 0} \frac{S(t)(x+h y)-S(t)x}{h}
\end{equation}
are well-defined with values in $X$.  Due to the $\sigma$-order completeness of the norm
one has
\begin{equation}\label{sprimeconv}
\bigg\| S'_\pm(t,x)y\mp \frac{S(t)(x\pm hy)-S(t)x}{h}\bigg\|\to 0\quad \text{as }h\downarrow 0.
\end{equation}
If the left and right directional derivatives coincide, then the directional derivative is continuous in time. More precisely, the following holds.
\begin{proposition}\label{prop:Sprime1}
Suppose that $S^\prime_+(t,x)y=S^\prime_-(t,x)y$ for some $x,y\in X$ and some $t\geq 0$. Then, the maps $[0,\infty)\to X$, $s\mapsto S^\prime_\pm(s,x)y$ are continuous at $t$. In particular, $\lim_{s\downarrow 0}S^\prime_\pm(s,x)y=y$.
\end{proposition}

\begin{proof}
 Since $S^\prime_-(s,x)y=-S^\prime_+(s,x)(-y)$ for all $s\geq 0$, it suffices to prove the continuity of the map $[0,\infty)\to X$, $s\mapsto S^\prime_+(s,x)y$ at $t$. For all $s\geq 0$ and $h>0$, let
 \[
   D_{h,\pm}(s,x)y:=\frac{S(s)(x\pm hy)-S(s)x}{\pm h}.
 \]
 By Corollary \ref{cor:cont}, the mapping $[0,\infty)\to X,\; s\mapsto D_{h,\pm}(s,x)y$ is continuous for all $h>0$. Let $\ep>0$. By \eqref{sprimeconv}, there exists some $h_\ep>0$ with
 \[
  \big\|D_{h_\ep,+}(t,x)y-S^\prime_+(t,x)y \big\|<\frac{\ep}{4}\quad \text{and}\quad \big\| D_{h_\ep,-}(t,x)y-S^\prime_-(t,x)y\big\|<\frac{\ep}{4}.
 \]
 Since the mapping $[0,\infty)\to X,\; s\mapsto D_{h_\ep,\pm}(s,x)y$ is continuous, there exists some $\de>0$ such that
 \[
  \big\|D_{h_\ep,+}(s,x)y-D_{h_\ep,+}(t,x)y\big\|<\frac{\ep}{4}\quad \text{and}\quad \big\|D_{h_\ep,-}(s,x)y-D_{h_\ep,-}(t,x)y\big\|<\frac{\ep}{4}
 \]
 for all $s\geq 0$ with $|s-t|<\de$. Hence,
 \begin{equation}\label{eq:Sprime1}
  \big\|D_{h_\ep,+} (s,x)y -S^\prime_+ (t,x)y\big\|<\frac{\ep}{2}\quad \text{and}\quad \big\|D_{h_\ep,-} (s,x)y-S^\prime_- (t,x)y\big\|<\frac{\ep}{2}
 \end{equation}
 for all $s\geq 0$ with $|s-t|<\de$. Since $S^\prime_-(s,x)y\leq S^\prime_+(s,x)y$, we obtain that
 \[
  S^\prime_+(s,x)y-S^\prime_-(t,x)y\geq S^\prime_-(s,x)y-S^\prime_-(t,x)y\geq D_{h_\ep,-}(s,x)y-S^\prime_-(t,x)y
 \]
 for all $s\geq 0$. On the other hand,
 \[
  S^\prime_+(s,x)y-S^\prime_+(t,x)y\leq D_{h_\ep,+}(s,x)y -S^\prime_+(t,x)y
 \]
 for all $s\geq 0$. Now, since $S^\prime_+(t,x)y=S^\prime_-(t,x)y$, we obtain that
 \[
  \big|S^\prime_+ (s,x)y-S^\prime_+(t,x)y\big|\leq \big|D_{h_\ep,+}(s,x)y-S^\prime_+(t,x)y\big|+\big|D_{h_\ep,-}(s,x)y- S^\prime_-(t,x)y\big|
 \]
 for all $s\geq 0$ and therefore, by \eqref{eq:Sprime1},
 \[
  \big\|S^\prime_+ (t,x)y-S^\prime_+(s,x)y\big\|<\ep
 \]
 for all $s\geq 0$ with $|s-t|<\de$. Since $S(0)=\id_X$ is linear, it follows that
 \[
  S^\prime_+(0,x)=S^\prime_-(0,x)=\id\nolimits_X
 \]
 and therefore, $\lim_{t\downarrow 0}S^\prime_\pm(t,x)y=S^\prime_\pm(0,x)y=y$.
\end{proof}

It is a straightforward application of Proposition \ref{domainlip} that $[0,\infty)\to X,\;t\mapsto S(t)x$ is locally Lipschitz continuous for all $x\in D(A)$. The following first main result states that it is even continuously differentiable on the domain.

\begin{theorem}\label{generatormain}
 Let $x\in D(A)$ and $t\geq 0$.
 \begin{itemize}
  \item[(i)] It holds $S(t)x\in D(A)$ with
  \[
    AS(t)x=S^\prime_+(t,x)Ax.
  \]
  If $S(t)$ is linear, this results in the well-known relation $AS(t)x=S(t)Ax$.
  \item[(ii)] For $t>0$, it holds
  \[
   \lim_{h\downarrow 0}\frac{S(t)x-S(t-h)x}{h}=S^\prime_-(t,x)Ax.
  \]
  \item[(iii)] It holds $S^\prime_+(t,x)Ax=S^\prime_-(t,x)Ax$. The mapping
  $[0,\infty)\to X$, $s\mapsto S(s)x$ is continuously differentiable and the derivative is given by
  \[
   \tfrac{{\rm d}}{{\rm d}s}S(s)x=AS(s)x=S^\prime_\pm (s,x)Ax\quad \text{for }s\geq 0.
  \]
  \item[(iv)] It holds
  \[
   S(t)x-x=\int_0^t AS(s)x\, \mathrm{d}s=\int_0^t S^\prime_+(s,x)Ax\, \mathrm{d}s=\int_0^t S^\prime_-(s,x)Ax\, \mathrm{d}s.
  \]
 \end{itemize}	
\end{theorem}

\begin{proof} (i) Let $t\geq 0$ and $(h_n)_n$ in $(0,\infty)$ with $h_n\downarrow 0$. Then,
  \[
   \frac{S(t+h_n)x-S(t)x}{h_n}-\frac{S(t)(x+h_nAx)-S(t)x}{h_n}=\frac{S(t)S(h_n)x-S(t)(x+h_nAx)}{h_n}.
  \]
  By Corollary \ref{loclipschitz}, there exist $L\geq 0$ and $r>0$ such that
  \[
   \|S(t)y-S(t)z\|\leq L\|y-z\|
  \]
  for all $y,z\in B(x,r)$. For $n\in \N$ sufficiently large, we thus obtain that
  \[
   \bigg\|\frac{S(t)S(h_n)x-S(t)(x+h_nAx)}{h_n}\bigg\|\leq L\bigg\|\frac{S(h_n)x-x}{h_n}-Ax\bigg\|\to 0.
  \]
  Since, by \eqref{sprimeconv},
  \[
   \frac{S(t)(x+h_nAx)-S(t)x}{h_n}\to S^\prime_+(t,x)Ax,
  \]
  we obtain the assertion.

 (ii) Let $t> 0$ and $(h_n)_{n\in\mathbb{N}}$ in $(0,t]$ with $h_n\downarrow 0$. Then,
  \[
   \frac{S(t)x-S(t-h_n)x}{h_n}-\frac{S(t)x-S(t)(x-h_nAx)}{h_n}=\frac{S(t)(x-h_nAx)-S(t-h_n)x}{h_n}.
  \]
  Again, by Corollary \ref{loclipschitz}, there exist $L\geq 0$ and $r>0$ such that
  \[
   \sup_{s\in [0,t]}\|S(s)y-S(s)z\|\leq L\|y-z\|
  \]
  for all $y,z\in B(x,r)$. By Corollary \ref{cor:crossconv}, we have $S(h_n)(x-h_nAx)\to x$. Hence, for $n\in \N$ sufficiently large, it follows that
  \[
   \bigg\|\frac{S(t-h_n)S(h_n)(x-h_nAx)-S(t-h_n)x}{h_n}\bigg\|\leq L\bigg\|\frac{S(h_n)(x-h_nAx)-x}{h_n}\bigg\|.
  \]
  Using Corollary \ref{cor:crossconv} and the convexity of $S_x$ and $S_{x-h_n Ax}$, we find that, for sufficiently large $n\in \N$,
  \begin{align*}
   \frac{S(h_n)(x-h_nAx)-x}{h_n}&=\frac{S_x(h_n)(-h_nAx)}{h_n}+\frac{S(h_n)x-x}{h_n}\\
   &\leq S_x(h_n)(-Ax)+\frac{S(h_n)x-x}{h_n}\to 0
  \end{align*}
  and
  \begin{align*}
   \frac{x-S(h_n)(x-h_nAx)}{h_n}&=\frac{S_{x-h_nAx}(h_n)(h_nAx)}{h_n}-\frac{S(h_n)x-x}{h_n}\\
   &\leq S_{x-h_nAx}(h_n)(Ax)-\frac{S(h_n)x-x}{h_n}\to 0.
  \end{align*}
  This shows that $\big\|\frac{S(h_n)(x-h_nAx)-x}{h_n}\big\|\to 0$, which implies that
  \[
   \bigg\|\frac{S(t)x-S(t-h_n)x}{h_n}-\frac{S(t)x-S(t)(x-h_nAx)}{h_n}\bigg\|\to 0.
  \]
  Since, by \eqref{sprimeconv},
  \[
   \frac{S(t)x-S(t)(x-h_nAx)}{h_n}\to S^\prime_-(t,x)Ax,
  \]
  we obtain the assertion.

 (iii) By definition, it holds $S^\prime_+(t,x)Ax\geq S^\prime_-(t,x)Ax$, and, for $t=0$,
  \[
   S^\prime_+(0,x)Ax=S^\prime_-(0,x)Ax=Ax.
  \]
  Therefore, let $t>0$ and $0<h\leq t$. Then, by convexity of $S_{S(t-h)x}$, for $h$ sufficiently small, it holds
  \begin{align*}
   \frac{S(t+h)x-S(t)x}{h}&=\frac{S(h)S(t)x-S(h)S(t-h)x}{h}\\
   &=\frac{S_{S(t-h)x}(h)\big(S(t)x-S(t-h)x\big)}{h}\\
   &\leq S_{S(t-h)x}(h)\bigg(\frac{S(t)x-S(t-h)x}{h}\bigg),
  \end{align*}
  which implies that
  \begin{align*}
   S^\prime_+(t,x)Ax&= A S(t) x = \lim_{h\downarrow 0} \frac{S(t+h)x-S(t)x}{h}\\
   &\leq \lim_{h\downarrow 0} S_{S(t-h)x}(h)\bigg(\frac{S(t)x-S(t-h)x}{h}\bigg)\\
   &=S^\prime_-(t,x)Ax,
  \end{align*}
  where we used Corollary \ref{cor:crossconv} and (ii) in the last step. Now, Proposition \ref{prop:Sprime1} yields that the mapping $[0,\infty)\to X$, $s\mapsto S^\prime_+(s,x)Ax$ is continuous.

  (iv) This follows directly from (iii) using the fundamental theorem of  calculus.
\end{proof}

As in the linear case, the generator of a convex $C_0$-semigroup is closed.
\begin{proposition}\label{prop:genclosed}
The generator $A$ is closed, i.e., for every sequence $(x_n)_{n\in\mathbb{N}}$ in $D(A)$ with $x_n\to x\in X$ and $Ax_n\to y\in X$, it holds $x\in D(A)$ and $Ax=y$.	
\end{proposition}

\begin{proof}
 First, notice that
 \[
  -S_{x_n}(s)(-Ax_n)\leq S^\prime_+(s,x_n)Ax_n\leq S_{x_n}(s)Ax_n,
 \]
 where we have used $S^\prime_+(s,x_n)Ax_n=S^\prime_-(s,x_n)Ax_n$ from Theorem \ref{generatormain} (iii), for all $s\geq 0$ and $n\in \N$. By Corollary \ref{loclipschitz}, there exist $L\geq 0$ and $r>0$ such that
 \[
  \sup_{s\in [0,1]}\|S(s) w-S(s)z\|\leq L\|w-z\|
 \]
 for all $w,z\in B(x\pm y,r)$. Hence, for $n\in \N$ sufficiently large,
 \begin{align*}
 \|S_{x_n}(s) A x_n-S_{x_n}(s)y\|  &\le L \|Ax_n-y\|\quad
\mbox{and}\\ \|S_{x_n}(s)(-A x_n)-S_{x_n}(s)(-y)\|&\le L \|Ax_n-y\|,
 \end{align*}
 so that
 \[
  \|S^\prime_+(s,x_n)Ax_n-y\|\leq 2L\|Ax_n-y\|+\|S_{x_n}(s)y-y\|+\|S_{x_n}(s)(-y)+y\|
 \]
 for all $s\in[0,1]$. By Theorem \ref{generatormain},
 \[
  \frac{S(h)x_n-x_n}{h}-y=\frac{1}{h}\int_0^h \big(S^\prime_+(s,x_n)Ax_n-y\big)\,\mathrm{d}s
 \]
for all $h>0$. Hence, for fixed $h\in (0,1]$, we find that
\begin{align*}
\bigg\|\frac{S(h)x-x}{h}-y\bigg\|&=\lim_{n\to \infty} \bigg\|\frac{S(h)x_n-x_n}{h}-y\bigg\|\leq \limsup_{n\to \infty} \frac{1}{h}\int_0^h \big\|S^\prime_+(s,x_n)Ax_n-y\big\|\, {\rm d}s\\
&\leq \limsup_{n\to \infty} 2L\|Ax_n-y\|+\sup_{0\leq s\leq h}\big(\|S_{x_n}(s)y-y\|+\|S_{x_n}(s)(-y)+y\|\big)\\
&=\sup_{0\leq s\leq h}\big(\|S_x(s)y-y\|+\|S_x(s)(-y)+y\|\big),
\end{align*}
where we used Corollary \ref{loclipschitz} in the last step. This shows that
\[
\bigg\|\frac{S(h)x-x}{h}-y\bigg\|\leq  \sup_{0\leq s\leq h}\big(\|S_x(s)y-y\|+\|S_x(s)(-y)+y\|\big)\to 0\quad \mbox{as }h\downarrow 0.
\]
That is, $x\in D(A)$ with $Ax=y$.
\end{proof}

The following theorem is the second main result of this section and shows uniqueness of the solution.

\begin{theorem}\label{thm:uniqueness1}
 Let $y\colon [0,\infty)\to X$ be a continuous function such that $y(t)\in D(A)$ for all $t\ge 0$, and
 \[
  \lim_{h\downarrow 0}\frac{y(t+h)-y(t)}{h}= Ay(t)\quad\text{for all }t\geq 0.
 \]
 Then, $y(t)=S(t)x$ for all $t\geq 0$, where $x:=y(0)$.		
\end{theorem}

\begin{proof}
Let $t>0$ and $g(s):=S(t-s)y(s)$ for all $s\in [0,t]$. Fix $s\in [0,t)$. For every $h>0$ with $h\leq t-s$, it holds
\begin{align*}
\frac{g(s+h)-g(s)}{h}&=\frac{S(t-s-h)y(s+h)-S(t-s)y(s)}{h}\\
&= \frac{S_{S(h)y(s)}(t-s-h)\big(y(s+h)-S(h)y(s)\big)}{h}.
\end{align*}
By Proposition \ref{sem:unibound}, there exist $L\geq 0$ and $r>0$ such that
\begin{equation}\label{rd01}
 \sup_{\tau\in [0,t]} \|S_x(\tau)z\|\leq L\|z\|
\end{equation}
for all $x\in B(y(s),r)$ and $z\in B(0,r)$. Hence, for $h$ sufficiently small, it follows that
\begin{align*}
 \bigg\|\frac{S_{S(h)y(s)}(t-s-h)\big(y(s+h)-S(h)y(s)\big)}{h}\bigg\|&\leq L\bigg\|\frac{y(s+h)-S(h)y(s)}{h}\bigg\|,
\end{align*}
where we used that $\lim_{h\downarrow 0} y(s+h)=y(s)=\lim_{h\downarrow 0}S(h)y(s)$. Since $y(s)\in D(A)$,
\begin{align*}
\frac{y(s+h)-S(h)y(s)}{h}&=\frac{y(s+h)-y(s)}{h}-\frac{S(h)y(s)-y(s)}{h}
\to Ay(s)-Ay(s)=0
\end{align*}
as $h\downarrow 0$. This shows that $\tfrac{g(s+h)-g(s)}{h}\to 0$ as $h\downarrow 0$.

We next show that the map $g\colon [0,t]\to X$ is continuous. Since its right derivative exists, it follows that $\lim_{h\downarrow 0}g(s+h)=g(s)$ for $s\in [0,t)$. Now, let $s\in (0,t]$ and $h>0$ sufficiently small. Then,
\begin{align*}
g(s-h)-g(s)&=S(t-s)S(h)y(s-h)-S(t-s)y(s)\\
&=S_{y(s)}(t-s)\big(S(h)y(s-h)-y(s)\big).
\end{align*}
Since $y(s-h)\to y(s)$ as $h\downarrow 0$, by Corollary \ref{cor:crossconv}, it follows that $S(h)y(s-h)\to y(s)$ as $h\downarrow 0$. Together with \eqref{rd01}, we obtain that $\lim_{h\downarrow 0}g(s-h)=g(s)$.

Finally, fix $\mu$ in the dual space $X^\prime$. Since $\mu g\colon [0,t]\to \mathbb{R}$ is continuous and its right derivative vanishes on $[0,t)$, it follows from \cite[Lemma 1.1, Chapter 2]{pazy2012semigroups} that $[0,t]\to X,\; s\mapsto \mu g(s)$ is constant. In particular, $\mu y(t)=\mu g(t)=\mu g(0)=\mu S(t)x$. This shows that $y(t)= S(t)x$, as $X^\prime$ separates the points of $X$.
\end{proof}

\begin{remark}
 With similar arguments as in the proof the previous theorem, one can show the following statement: Let $y\colon [0,\infty)\to X$ be a continuous function with $y(t)\in D(A)$ for all $t\ge 0$ and
 $\lim_{h\downarrow 0}\tfrac{y(t)-y(t-h)}{h}=Ay(t)$ for all $t>0$. Then, $y(t)=S(t)x$ for all $t\geq 0$ with $x:=y(0)$.
\end{remark}

Theorem~\ref{thm:uniqueness1} implies that convex semigroups are determined by their generators whenever the domain is dense.
\begin{corollary}\label{cor:uniqueness1}
 Let $T$ be a convex $C_0$-semigroup with generator $B\subset A$, i.e., $D(B)\subset D(A)$ and $A|_{D(B)}=B$. If $\overline{D(B)}=X$, then $S(t)=T(t)$ for all $t\geq 0$.
\end{corollary}

\begin{proof}
	For every $x\in D(B)$, the mapping $[0,\infty)\to X$, $t\mapsto T(t)x$ satisfies the assumptions of Theorem \ref{thm:uniqueness1}. Indeed, $[0,\infty)\to X,\;t\mapsto T(t)x$ is continuous by Corollary \ref{cor:cont}, and, by Theorem \ref{generatormain}, $T(t)x\in D(B)\subset  D(A)$ for all $t\geq 0$
	with
	\[
	\lim_{h\downarrow 0}\frac{T(t+h)x-T(t)x}{h}=\lim_{h\downarrow 0}\frac{T(h)T(t)x-T(t)x}{h}=BT(t)x=AT(t)x.
	\]
By Theorem \ref{thm:uniqueness1}, it follows that $T(t)x=S(t)x$ for all $t\ge 0$. Finally, since, by Corollary \ref{cor:Lip}, the bounded convex functions $T(t)$ and $S(t)$ are continuous and $\overline{D(B)}=X$, it follows that $S(t)=T(t)$ for all $t\ge 0$.	
\end{proof}

\begin{corollary}\label{cor:wellposed}
The abstract Cauchy problem
\[
 {\rm (CP)}\qquad \begin{cases}y'(t)=Ay(t),& \text{for all }t\geq 0,\\
  \,y(0)=x
 \end{cases}
\]
is \textit{(classically) well-posed} in the following sense:
\begin{itemize}
 \item[(i)] For all $x\in D(A)$, {\upshape (CP)} has a unique classical solution $y\in C^1([0,\infty);X)$ with $y(t)\in D(A)$ for all $t\geq 0$ and $Ay\in C([0,\infty);X)$.
 \item[(ii)] For all $x_0\in D(A)$ and $T>0$, there exist $L\geq 0$ and $r>0$ such that
 \[
  \sup_{t\in [0,T]}\|y(t,x)-y(t,z)\|<L\|x-z\| \quad \text{for all }x,z\in D(A)\cap B(x_0,r),
 \]
 where $y(\, \cdot\, , x)$ denotes the unique solution to {\upshape (CP)} with initial value $x\in D(A)$.
 \item[(iii)] For all $t>0$ and $r>0$, there exists some constant $C\geq 0$ such that
 \[
  \|y(t,x)\|\leq C\quad \text{for all }x\in D(A)\text{ with }\|x\|\leq r.
 \]
\end{itemize}
\end{corollary}

\begin{proof}
 By Theorem \ref{generatormain} and Theorem \ref{thm:uniqueness1}, it follows that, for every $x\in D(A)$, the Cauchy problem (CP) has a unique classical solution $y\in C^1([0,\infty);X)$ such that $y(t)\in D(A)$ for all $t\geq 0$ and $Ay\in C([0,\infty);X)$, and which is given by $y(t) = S(t)x$. By Corollary \ref{loclipschitz}, we obtain (ii), and (iii) is the boundedness of the operator $S(t)$.
\end{proof}

\begin{remark}
 Assume that for some operator $A_0\colon D(A_0)\subset X\to X$  the abstract Cauchy problem is well-posed in the sense of Corollary \ref{cor:wellposed}. Let the domain $D(A_0)$ be a dense linear subspace of $X$, and assume that
 the map $D(A_0)\to X,\;x\mapsto y(t,x)$ is convex for all $t\geq 0$. Then, there exists a unique convex $C_0$-semigroup $S=(S(t))_{t\in [0,\infty)}$ with $S(t)x=y(t,x)$ for all $x\in D(A_0)$. Moreover, $A_0\subset A$, where $A$ is the generator of $S$, and $D(A_0)$ is $S(t)$-invariant for all $t\geq 0$, i.e., $S(t)x\in D(A_0)$ for all $t\geq 0$ and $x\in D(A_0)$.

 In fact, we can define the operator $S(t)x:=y(t,x)$ for all $t\geq 0$ and $x\in D(A_0)$. As $S(t)$ is bounded by (iii) and convex, it is Lipschitz on bounded subsets of $D(A_0)$ by Corollary \ref{cor:Lip}. Therefore, there exists a unique continuous extension $S(t)\colon X\to X$, which again is bounded and convex. By the uniqueness in (i), the semigroup property for the family $S=(S(t))_{t\in [0,\infty)}$ holds for all $x\in D(A_0)$, and therefore for all $x\in X$.
 Similarly, the strong continuity follows by $y(\cdot,x)\in C([0,\infty); X)$ for $x\in D(A_0)$ and (ii). Finally, as, for every $x\in D(A_0)$, the function $y(\cdot,x)$ is differentiable at zero with derivative $Ax$, we obtain $D(A_0)\subset D(A)$ with $A|_{D(A_0)} = A_0$ as well as, by (i), the invariance of $D(A_0)$ under $S(t)$.

 In this way, we can construct a convex $C_0$-semigroup by solving the Cauchy problem only for initial values $x\in D(A_0)$. In applications, one might have $D(A_0)$ being much smaller than $D(A)$.
\end{remark}

\begin{remark}
 We would like to point out that very little can be said about structural properties of the domain $D(A)$ when $S$ is nonlinear. If $S$ is sublinear, the generator and the domain scale with positive multiples, i.e., $\la x\in D(A)$ with $A(\la x)=\la Ax$ for all $x\in D(A)$ and $\la\geq 0$, which is a direct consequence of positive homogeneity of the semigroup. Although in typical situations, when $X=L^p$, the domain contains a dense subspace, which, in most applications, is the space $C_c^\infty$ of all smooth functions with compact support, the density of the domain, in a general setting, remains an open question. Considering the semigroup envelope for two linear semigroups for which the intersection of the domains consists of only $0$, suggests that the domain should fail to be dense in general. A less pathological case which is not covered by the setup in this section is given by the semigroup envelope for a family of heat semigroups with varying covariance operator on the space of all bounded uniformly continuous functions on a separable Hilbert space $H$ as in \cite{dkn2} and \cite{roecknen}, and also suggests that the domain is typically not dense.
\end{remark}

\section{Semigroup envelopes}\label{sec.envelopes}
As in the previous section, we assume that $X$ has an order continuous norm implying that $X$ is Dedekind super complete (see beginning of Section~\ref{sec:dedcomplete}). For two semigroups $S$ and $T$ on $X$, we write $S\leq T$ if
\[
 S(t)x\leq T(t)x\quad \text{for all }t\geq 0\text{ and }x\in X.
\]
We would like to point out that our definition of dominance for two semigroups is not consistent with the notion of dominance for linear semigroups. If $S$ and $T$ are both linear, $S\leq T$ implies that $S=T$. Our definition of $\leq$ is therefore only nontrivial, in the sense that it is a strict inequality, if $S$ or $T$ are nonlinear.

Throughout this section, let $(S_\lambda)_{\lambda\in\Lambda}$ be a family of convex monotone semigroups on $X$. We say that a semigroup $S$ is an \textit{upper bound} of $(S_\lambda)_{\lambda\in\Lambda}$ if $S\geq S_\la$ for all $\la\in \La$.

\begin{definition}
 We call a semigroup $S$ (if existent) the \textit{(upper) semigroup envelope} of $(S_\lambda)_{\lambda\in\Lambda}$ if it is the smallest upper bound of $(S_\lambda)_{\lambda\in\Lambda}$, i.e. if $S$ is an upper bound of $(S_\lambda)_{\lambda\in\Lambda}$ and $S\leq T$ for any other upper bound $T$ of $(S_\lambda)_{\lambda\in\Lambda}$.
\end{definition}

Notice that the definition of a semigroup envelope already implies its uniqueness. However, the existence of a semigroup envelope is not given in general. In \cite{dkn2} and \cite{roecknen} the existence of a semigroup envelope, under certain conditions, has been shown for families of semigroups on spaces of uniformly continuous functions. This is done following an idea of Nisio \cite{MR0451420}, who was, to the best of our knowledge, the first to investigate the existence of semigroup envelopes. A related construction is the one of a modulus for linear $C_0$-semigroups by Becker and Greiner \cite{MR868254}. It was shown (cf.~\cite{dkn2},\cite{roecknen},\cite{MR0451420}) that, for $C_0$-semigroups, there is a relation between the semigroup envelope, that is the supremum, of a family of semigroups and the pointwise supremum of their generators. In this subsection, we now want to show that the construction of Nisio, which is a pointwise optimization on a finer and finer time-grid, can be realized on Dedekind super complete Banach lattices. Moreover, we show that the ansatz proposed by Nisio is in fact the only way to construct the supremum of a family of semigroups. We further show that, under certain conditions, the semigroup envelope is a convex monotone $C_0$-semigroup, which makes the results from the previous subsection applicable. In view of the examples in \cite{dkn2} and \cite{roecknen}, this could be the starting point for $L^p$-semigroup theory for a class of Hamilton-Jacobi-Bellman equations.

In the sequel, we consider finite partitions $P:= \{\pi\subset[0,\infty): 0\in\pi, \, \pi\text{ finite}\}$. For a partition $\pi=\{t_0,t_1,\dots,t_m\}\in P$ with $0=t_0< t_1< \ldots < t_m$, we define $|\pi|_\infty := \max_{j=1,\dots,m} (t_j-t_{j-1})$, and we set $|\{0\}|_\infty=0$. The set of partitions with end-point $t$ is denoted by $P_t$, i.e., $P_t := \{\pi \in P: \max \pi = t\}$.

Assume that the set $\{S_\la (t)x\colon  \la \in \La\}$ is bounded above for all $x\in X$ and all $t>0$. Let $x\in X$. Then, we set
\[
 J_h x:= \sup_{\la \in \La }S_\la (h)x
\]
for all $h>0$ and
\[
 J_\pi x:=J_{t_1-t_0}\cdots J_{t_m- t_{m-1}} x
\]
for any partition $\pi=\{t_0,t_1,\dots,t_m\}\in P$ with $0=t_0< t_1< \ldots < t_m$. Notice that, for $x\in X$ and $h_1,h_2\geq 0$,
 \[
  S_\lambda(h_1+h_2)x=S_\lambda(h_1) S_\lambda(h_2)x\leq J_{h_1}J_{h_2}x
 \]
 for all $\lambda\in \Lambda$, which implies that $J_{h_1+h_2}x\leq  J_{h_1}J_{h_2}x$. In particular,
 \begin{equation}\label{eq.finerpart}
  J_{\pi_1}x\leq J_{\pi_2}x
 \end{equation}
 for all $x\in X$ and $\pi_1,\pi_2\in P$ with $\pi_1\subset \pi_2$.

\begin{theorem}\label{thm:envelope}
 Assume that, for all $t\geq 0$, there is a bounded operator $C(t)\colon X\to X$ with $J_\pi x\leq C(t)x$ for all $\pi \in P_t$ and $x\in X$.~Then, the semigroup envelope $S=(S(t))_{t\in [0,\infty)}$ of $(S_\lambda)_{\lambda\in\Lambda}$ exists, is a convex monotone semigroup, and is given by
 \begin{equation}\label{envelope}
  S(t)x=\sup_{\pi \in P_t}J_\pi x
 \end{equation}
 for all $t\geq 0$ and $x\in X$. If $C(t)x\to x$ as $t\downarrow 0$ for all $x\in X$ and $S_{\la_0}$ is a $C_0$-semigroup for some $\la_0\in \La$, then $S$ is a $C_0$-semigroup.~Moreover, if $S_\la$ is sublinear for all $\la\in \La$, then the semigroup envelope $S$ is sublinear.
\end{theorem}

\begin{proof}
 Clearly, we have that $S_\la (h)x\leq J_h x$ for all $\la \in \La$, $h>0$ and all $x\in X$. Moreover, since $S_\la$ is montone and convex for all $\la\in \La$, it follows that $J_h$ is monotone and convex for all $h\geq 0$. Consequently, $J_\pi$ is monotone and convex with $S_\la (t)x\leq J_\pi x \leq C(t)x$ for all $\la \in \La$, $t\geq 0$, $\pi\in P_t$ and $x\in X$, showing that $S=(S(t))_{t\geq 0}$, given by \eqref{envelope}, is well-defined, monotone, convex and an upper bound of the family $(S_\la)_{\la\in \La}$. Moreover, one directly sees that $S$ is sublinear as soon as all $S_\la$ are sublinear. From
 \[
  S_{\la_0} (t)x\leq S(t)x\leq C(t)x \quad \text{and}\quad S_{\la_0} (t)x-x\leq S(t)x-x\leq C(t)x-x,
 \]
 it follows that
 \[
  \|S(t)x\|\leq \|S_{\la_0} (t)x\|+\|C(t)x\|
 \]
 and
 \[
  \| S(t)x-x\|\leq \| S_{\la_0} (t)x-x\|+\|C(t)x-x\|
 \]
 for all $t\geq 0$, $x\in X$ and some $\la_0\in \La$. This implies that $S(t)$ is bounded for all $t\geq 0$, and that $\lim_{t\downarrow 0}S(t)x=x$ as soon as $C(t)x\to x$ as $t\downarrow 0$ and $S_{\la_0}$ is a $C_0$-semigroup for some $\la_0\in \La$. Next, we show that $S=(S(t))_{t\geq 0}$, defined by \eqref{envelope}, is a semigroup. Clearly, $S(0)x=x$ for all $x\in X$. In order to show that $S(t+s)=S(t)S(s)$ for all $s,t\geq 0$, let $s,t\geq 0$ and $x\in X$. Then, it is easily seen that $S(t+s)x\leq S(t)S(s)x$ since, by Equation \eqref{eq.finerpart}, for all $\pi\in P_{t+s}$,
 \[
  J_\pi x\leq J_{\pi_0}J_{\pi_1}x,
 \]
 where $\pi_0:=\{u\in \pi : u\leq t\}\cup\{t\}$ and $\pi_1:=\{u-t: u\in \pi, u\geq t\}\cup \{0\}$. On the other hand, there exists a sequence $(\pi_n)_n$ in $P_s$ with $S(s)x=\sup_{n\in \N}J_{\pi_n}x$. Defining
 \[
  \pi_n^*:=\bigcup_{k=1}^n \pi_k
 \]
for all $n\in \N$, we obtain that $J_{\pi_n^*}x\to S(s)x$, by the $\sigma$-order continuity of the norm. Consequently,
\[
 J_\pi S(s)x=\lim_{n\to\infty} J_\pi J_{\pi_n^*}x\leq S(t+s)x
\]
for all $\pi\in P_t$, where, in the first equality, we used the fact that $J_\pi$ is continuous since, by Lemma \ref{lemma1}, it is convex and bounded. Taking the supremum over all $\pi\in P_t$, we obtain that $S(t)S(s)x\leq S(t+s)x$.

Finally, let $T$ be an upper bound of $(S_\la)_{\la\in \La}$. Then, $J_h x \leq T(h)x$ for all $h>0$ and all $x\in X$ and consequently $J_\pi x \leq T(t)x$ for all $t\geq 0$, $\pi\in P_t$ and $x\in X$, which shows that $S(t)x\leq T(t)x$ for all $t\geq 0$ and $x\in X$.
\end{proof}

\begin{remark}
The proof of Theorem \ref{thm:envelope} shows that under the additional assumption that $X$ is a KB-space (cf.\ \cite[Chapter 7]{wnuk1999banach}), i.e., a Banach lattice in which every norm bounded increasing net in $X$ is norm convergent, the existence of the semigroup envelope can be established under the slightly weaker condition
\begin{equation}\label{eq.normjpi}
\sup_{\pi\in P_t}\|J_\pi x\| <\infty
\end{equation}
for all $x\in X$ and $t\geq 0$, instead of $J_\pi x\leq C(t)x$ for all $\pi \in P_t$ and $x\in X$. A condition ensuring the strong continuity in this case would be
\begin{equation}\label{eq.normjpi2}
 \sup_{\pi\in P_t}\|J_\pi x-x\| \to 0\quad \text{as } t\to 0
\end{equation}
for all $x\in X$ instead of $C(t)x\to x$ as $t\to 0$. Although every $L^p$-space, for $p\in [1,\infty)$, is a KB-space (cf. \cite[Corollary 2.4.13]{MR1128093}), \eqref{eq.normjpi2} is usually not a very handy condition, and the pointwise estimate in terms of $C(t)$ gives additional possibilities to verify the strong continuity, see, for instance, Theorem \ref{thm:lpenvelope} (ii), below.
\end{remark}

\begin{corollary}
 Let the semigroup $T$ be an upper bound of the family $(S_\lambda)_{\lambda\in\Lambda}$. Then, the semigroup envelope of $(S_\lambda)_{\lambda\in\Lambda}$ exists and is given by
 \eqref{envelope}. If $T$ is a $C_0$-semigroup and $S_{\la_0}$ is a $C_0$-semigroup for some $\la_0 \in \La$, then $S$ is a $C_0$-semigroup.
\end{corollary}

\begin{proof}
 As we saw in the proof of the previous theorem, $S_\la (t)x\leq J_\pi x \leq T(t)x$ for all $\la \in \La$, $t\geq 0$, $\pi\in P_t$ and $x\in X$. Therefore, the upper bound $C(t)$ in the previous theorem can be chosen to be $T(t)$.
\end{proof}

\begin{corollary}
 Let $S$ be the semigroup envelope of the family $(S_\lambda)_{\lambda\in\Lambda}$. Then,
 \[
  S(t)x=\sup_{\pi \in P_t}J_\pi x
 \]
 for all $t\geq 0$ and $x\in X$.
\end{corollary}

\section{Convolution semigroups on $L^p$}\label{sec:convsglp}

Let $d\in \N$. In \cite{dkn2}, the semigroup envelope, discussed in the previous section, has been constructed for a wide class of L\'evy processes. In \cite[Example 3.2]{dkn2}, the authors consider families $(S_\la)_{\la\in \La}$ of linear semigroups on the space $\BUC=\BUC(\R^d)$ of bounded uniformly continuous functions, which are indexed by a L\'evy triplet $\la=(b,\Si,\mu)$. Recall that a L\'evy triplet $(b,\Si,\mu)$ consists of a vector $b\in \R^d$, a symmetric positive semidefinite matrix $\Si\in \R^{d\times d}$ and a L\'evy measure $\mu$ on $\R^d$. For each L\'evy triplet $\la$, the semigroup $S_\la$ is the one generated by the transition kernels of a L\'evy process with L\'evy triplet $\la$. More precisely,
\begin{equation}\label{levsg}
 \big(S_\la (t)f\big)(x):=\E\big[f(x+L_t^\lambda)\big]
\end{equation}
for $t\geq 0$, $f\in \BUC$ and $x\in \R^d$, where $L_t^\lambda$ is a L\'evy process on a probability space $(\Om,\FF,\P)$ with L\'evy triplet $\la$. In \cite[Example 3.2]{dkn2}, it was shown that, under the condition
\begin{equation}\label{levcond}
 \sup_{(b,\Si,\mu)\in \La} |b|+|\Si|+\int_{\R^d\setminus \{0\}}1\wedge |x|^2\, {\rm d}\mu(x)<\infty,
\end{equation}
the semigroup envelope $S_{\BUC}$ for the family $(S_\la)_{\la\in \La}$ exists and that in this case (cf. \cite[Lemma 5.10]{dkn2})
\begin{equation}\label{eq:bucgen}
 \lim_{h\downarrow 0}\bigg\|\frac{S_{\BUC}(h)f-f}{h}-\sup_{\la\in \La}A_\la f\bigg\|_\infty=0\quad \text{for }f\in \BUC^2.
\end{equation}
Here, $\BUC^2=\BUC^2(\R^d)$ is the space of all twice differentiable functions with bounded uniformly continuous derivatives up to order $2$ and $A_\la$ is the generator of the semigroup $S_\la$ for each $\la\in \La$. Notice that the setup in \cite{dkn2} is not contained in the setup of the previous subsection since $\BUC$ is not Dedekind super complete and does not possess a $\sigma$-order continuous norm. Recall that, for each L\'evy triplet $\la$, \eqref{levsg} also gives rise to a linear monotone $C_0$-semigroup on $L^p=L^p(\R^d)$, which will again be denoted by $S_\la$ (cf. \cite[Theorem 3.4.2]{MR2512800}). Therefore, the question arises if under a similar condition as \eqref{levcond}, the semigoup envelope of the family $(S_\la)_{\la\in \La}$ can be constructed on $L^p$. In general, the answer to this question is negative as the following example shows.

\begin{example}[Uncertain shift semigroup]\label{ex:robustshift}
 Let $d=1$ and $(S_\la (t)f)(x):=f(x+t\la)$ for $\la\in \La:=[-1,1]$, $t\geq 0$, $f\in L^p(\R)$ and $x\in \R$. Then, for $f\in L^p(\R)$ given by $f(x)=|x|^{-1/2p}1_{[-1,1]}(x)$,
 \[
 \sup_{\la\in \La} (S_\la(t)f)(x)=\infty\quad \text{for all }t\geq 0\text{ and }x\in [-t,t].
 \]
 Therefore, the set $\{S_\la(t)f\colon \la\in \La\}$ does not have a least upper bound in $L^p$ for all $t>0$. In particular, the semigroup envelope of the family $(S_\la)_{\la\in \La}$ does not exist although the set $\La$ satisfies  condition \eqref{levcond}.
\end{example}

 In view of the previous example, additional conditions are required in order to guarantee the existence of the semigroup envelope on $L^p$. In the sequel, let $C_c^\infty$ denote the space of all $C^\infty$-functions $f\colon \R^d\to \R$ with compact support $\supp f$.

\begin{theorem}\label{thm:lpenvelope}
 Let $\La$ be a non-empty set of L\'evy triplets that satisfies \eqref{levcond}.
 \begin{itemize}
  \item[(i)] Assume that, for each $t>0$, there exists a bounded operator $C(t)\colon L^p\to L^p$ with
 \begin{equation}\label{eq:levcon2}
  J_\pi f\leq C(t)f\quad \text{for all }t> 0,\; \pi\in P_t\text{ and }f\in L^p.
 \end{equation}
 Then, the semigroup envelope $S$ of $(S_\la)_{\la\in \La}$ exists, and is a monotone sublinear semigroup.
 \item[(ii)] In addition to \eqref{eq:levcon2}, assume that
 \begin{equation}\label{eq:levcon3}
  \sup_{\la\in \La} A_\la f\in L^p\quad \text{for all }f\in C_c^\infty
 \end{equation}
 and that, for every $f\in C_c^\infty$ and every $\ep>0$, there exists a compact set $K\subset \R^d$ with $\supp f\subset K$ and
 \begin{equation}\label{eq:levcon4}
  \limsup_{h\downarrow 0}\bigg(\int_{\R^d\setminus K} \frac{\big|\big(C(h)f\big)(x)\big|^p}{h}\, {\rm d}x\bigg)^{1/p}\leq \ep.
 \end{equation}
 Then, the semigroup $S$ is a $C_0$-semigroup, $C_c^\infty\subset D(A)$ and
 \[
 Af=\sup_{\la\in \La}A_\la f
 \]
 for all $f\in C_c^\infty$, where $A$ denotes the generator of $S$.
 \end{itemize}
\end{theorem}

\begin{proof}
 (i) By Theorem \ref{thm:envelope}, it is clear that \eqref{eq:levcon2} implies the existence of the semigroup envelope $S$ and that the latter is monotone and sublinear.

 (ii) Let $f\in C_c^\infty$. We show that $f\in D(A)$ with $Af=\sup_{\la\in \La}A_\la f=:Bf$. Let $\ep>0$. By \eqref{eq:levcon3} and \eqref{eq:levcon4}, there exists some compact set $K\subset \R^d$ with $\supp f\subset K$ and
 \[
  \bigg(\int_{\R^d\setminus K} \big|\big(Bf\big)(x)\big|^p {\rm d}x\bigg)^{1/p}\le \frac{\ep}{4}\quad \text{and}\quad \bigg(\int_{\R^d\setminus K} \frac{\big|\big(C(h)g\big)(x)\big|^p}{h}\, {\rm d}x\bigg)^{1/p}\le \frac{\ep}{4}
 \]
 for $g=f,-f$ and $h>0$ sufficiently small. Since $f\in C_c^\infty\subset \BUC^2\cap L^p$, it follows that $S(t)f=S_{\BUC}(t)f$ for all $t\geq 0$. Hence, by \eqref{eq:bucgen},
\begin{align*}
 \bigg\|\frac{S(h)f-f}{h}-Bf\bigg\|_p&\leq {\rm vol}(K)^{1/p} \bigg\|\frac{S(h)f-f}{h}-Bf\bigg\|_\infty +\bigg(\int_{\R^d\setminus K} \big|\big(Bf\big)(x)\big|^p {\rm d}x\bigg)^{1/p}\\
 &\quad +\bigg(\int_{\R^d\setminus K}\frac{\big|\big(S(h)f\big)(x)\big|^p}{h}\, {\rm d}x\bigg)^{1/p}\le \ep
\end{align*}
 for $h>0$ sufficiently small, where ${\rm vol}(K)$ denotes the Lebesgue measure of $K$.

 In particular, $\|S(h)f-f\|_p\to 0$ for all $f\in C_c^\infty$.  Since $C_c^\infty$ is dense in $L^p$ and $S(t)\colon L^p\to L^p$ is continuous, this implies the strong continuity of $S$.
\end{proof}

Notice that the semigroup envelope from the previous theorem is exactly the extension of the semigroup envelope on $\BUC$, constructed in \cite{dkn2}, to the space $L^p$. More precisely, for each $t\geq 0$, the operator $S(t)$ is the unique bounded monotone sublinear operator $L^p\to L^p$ with $S(t)f=S_{\BUC}(t)f$ for all $f\in \BUC\cap L^p$. We will now give two examples of L\'evy semigroups $(S_\la)_{\la\in \La}$, where the semigroup envelope exists on $L^p$. The first one is a semilinear version of Example \ref{ex:robustshift}. The problem in Example \ref{ex:robustshift} arises due to shifting sufficiently integrable poles. In order to treat this problem, one first has to smoothen a given function $f\in L^p$ via a suitable normal distribution and then shift the smooth version of $f$. This results in the following example.

\begin{example}[$g$-expectation]\label{ex:smallg}
Let $d\in \N$, $p\in [1,\infty)$, and
\[
 \ph_\la(t,x):=(2\pi t)^{-d/2}e^{-\frac{|x+\la t|^2}{2t}} \quad \text{for }\la,x\in \R^d\text{ and }t> 0.
\]
For $\la\in \R^d$, we consider the linear $C_0$-semigroup $S_\la=(S_\la (t))_{t\in [0,\infty)}$ in $L^p = L^p(\R^d)$ given by
$S_\lambda(0)f= f$ and
\[
 \big(S_\la (t)f\big)(x):=\int_{\R^d}f(y)\ph_\la(t,x-y)\, {\rm d}y=\big(f\ast \ph_\la(t,\,\cdot\,)\big)(x)=\E\big[f(x+W_t+\la t )\big]
\]
for all $t> 0$, $f\in L^p$ and $x\in \R^d$, where $(W_t)_{t\in [0,\infty)}$ is a $d$-dimensional Brownian motion on a probability space $(\Om,\FF,\P)$. For each $\lambda\in \Lambda$, the generator $A_\lambda$ of $S_\lambda$ is given by
$D(A_\lambda) = W^{2,p}$ and
\[ A_\lambda f = \tfrac12 \Delta f + \lambda\cdot\nabla f\quad \text{for }f\in W^{2,p},\]
where $\Delta$ denotes the Laplacian,  `$\,\cdot\,$' is the scalar product in $\R^d$, and $W^{2,p} = W^{2,p}(\R^d)$ stands for the $L^p$-Sobolev space of order 2; see also \cite[Theorem~3.1.3]{Lunardi95} for the generation of a $C_0$-semigroup in $L^p$ and \cite[Theorem~31.5]{Sato99} for the connection between generator and L\'evy triplet. Now, let $\La\subset \R^d$ be bounded and non-empty, and define
\begin{equation}\label{eq:gexp1}
 \big(J_h f\big)(x):=\sup_{\la\in \La}\big(S_\la (h)f\big)(x)\quad \text{for }h\geq 0,\; f\in L^p\text{ and }x\in \R^d.
\end{equation}
Notice that, for $h>0$, $S_\la(h)f\in \BUC$ for all $f\in L^p$, which is why the supremum in \eqref{eq:gexp1} can be understood pointwise for $h>0$.

We show that the conditions of Theorem~\ref{thm:lpenvelope} are satisfied.
For the construction of an upper bound, we use the relation
\[
 \ph_\la(h,x-y)=e^{-\la \cdot (x-y)-h|\la|^2/2}\ph_0(h,x-y)
\]
for all $\la\in \R^d$, $h>0$ and $x,y\in \R^d$. With this and H\"older's inequality, it follows that
\begin{align*}
 \big(J_h f\big)(x)
 &= \sup_{\la\in \La}\int_{\R^d}f(y)e^{-\la \cdot (x-y)-h|\la|^2/2}\ph_0(h,x-y)\, {\rm d}y\\
 &= \sup_{\la\in \La}\E\Big[f(x+W_h)e^{-\la \cdot W_h-h|\la|^2/2}\Big]\\
 &\leq \Big(\E\big[|f(x+W_h)|^p\big]\Big)^{1/p}\sup_{\la\in \La}\Big(e^{-qh|\la|^2/2}\E\big[e^{-q\la \cdot W_h}\big]\Big)^{1/q}\\
 & =\Big(\E\big[|f(x+W_h)|^p\big]\Big)^{1/p}\sup_{\la\in \La} e^{(q-1)h|\la|^2/2}\\
 &= \Big(\E\big[|f(x+W_h)|^p\big]\Big)^{1/p} e^{(q-1)h\overline\la^2/2}=:\big(C(h)f\big)(x),
\end{align*}
where $\overline \la:=\sup_{\la\in \La}|\la|$ and $\frac 1p + \frac 1q =1$. As
\[ \big[(C(h)f)(x)\big]^p = e^{qh\overline{\lambda}^2/2} \big[ |f|^p\ast \ph_0(h,\,\cdot\, )\big](x),\]
we obtain that $C(h_1)C(h_2) = C(h_1+h_2)$ for $h_1,h_2>0$. Therefore,
\[
 J_\pi f\leq C(t_1-t_0)\cdots C(t_m-t_{m-1})f=C(t_m)f
\]
for any partition $\pi=\{t_0,t_1,\dots,t_m\}\in P$ with $0=t_0< t_1< \ldots < t_m$. By Fubini's theorem,
\[
\|C(h)f\|_p^p=e^{qh\overline{\lambda}^2/2}\int_{\R^d}\int_{\R^d} |f(x-y)|^p \ph_0(h,y)\, {\rm d}y\, {\rm d}x=e^{qh\overline{\lambda}^2/2}\|f\|_p^p
\]
for all $h>0$ and $f\in L^p$, showing that $C(h)\colon L^p\to L^p$ is bounded.

Now, let $f\in C_c^\infty$. We consider
\begin{equation}
  \label{rd02}
  (Bf)(x) := \sup_{\lambda\in\Lambda} (A_\lambda f)(x) = \tfrac12 \Delta f(x) + \sup_{\lambda\in \Lambda} \lambda\cdot\nabla f(x)
\end{equation}
for $x\in \R^d$. As, for every $\lambda\in\Lambda$ and $x\in\R^d$,
\[ |\lambda\cdot\nabla f(x)|\le \sum_{j=1}^d |\lambda_j|\,|\partial_j f(x)| \le \overline\lambda \sum_{j=1}^d |\partial_j f(x)| ,\]
we obtain
\begin{equation}\label{rd03}
 \|Bf\|_{L^p} \le C \big( \|\Delta f\|_{L^p} + \overline\lambda\|\nabla f\|_{L^p(\R^d;\R^d)}\big) \le C \max\{1,\overline\lambda\} \|f\|_{W^{2,p}},
\end{equation}
with a constant $C$ independent of $f$ and $\Lambda$, which shows, in particular, that $Bf\in L^p$ for all $f\in C_c^\infty$.

It remains to verify \eqref{eq:levcon4}. Let $f\in C_c^\infty$, and choose a compact set $K\subset\R^d$  with $\{ x+y: x\in \supp f,\, |y|\le 1\}\subset K$. For $x\in \R^d\setminus K$, we obtain $f(x+W_h)=0$ if $|W_h|\le1$, and therefore,
\[ \big( |f|^p\ast \ph_0(h,\,\cdot\,)\big)(x) = \E\big( |f(x+W_h)|^p\big) = \E \big( \mathbf{1}_{\{|W_h|>1\}}|f(x+W_h)|^p  \big).\]
By Fubini's theorem and Markov's inequality, for any $s>2$,
\begin{align*}
 \frac{1}{h}&\int_{\R^d\setminus K}\E\big( \mathbf{1}_{\{|W_h|>1\}}|f(x+W_h)|^p\big)\, {\rm d}x  = \frac 1h \E \Big[ \mathbf{1}_{\{|W_h|>1\}}\int_{\R^d\setminus K} |f(x+W_h)|^p\, {\rm d}x\Big]\\
 & \leq \frac{1}{h}\|f\|_p^p \;\P(|W_h|>1) = \frac{1}{h}\|f\|_p^p \;\P\big(|W_1|>h^{-1/2}\big)\leq h^{s/2-1}\E\big[|W_1|^s\big]\to 0
\end{align*}
as $h\downarrow 0$.
By definition of $C(h)$, it follows that $\frac{1}{h}\int_{\R^d\setminus K}\big|\big(C(h)f\big)(x)\big|^p\, {\rm d}x\to 0$ as $h\downarrow 0$.
We have seen that all conditions of Theorem~\ref{thm:lpenvelope} are satisfied, and therefore
the semigroup envelope $S=(S(t))_{t\in [0,\infty)}$ of $(S_\la)_{\la\in \La}$ exists, and is a sublinear monotone $C_0$-semigroup. In particular, we obtain a unique classical solution to the Cauchy problem
\begin{equation}\label{eq.cauchygexp}
 u'(t)=Au(t)\quad \text{for all }t\geq 0\quad u(0)=f
\end{equation}
in the sense of Corollary~\ref{cor:wellposed} for all initial values $f\in D(A)$, where $A$ is the generator of $S$.\\

As the map $\R^d\to \R,\, x\mapsto \sup_{\lambda\in\Lambda} \lambda\cdot x$ is Lipschitz (which follows, e.g., by Lemma~\ref{A.6}), the same holds for the nonlinearity
\[ F\colon W^{1,p} \to L^p,\, f\mapsto \sup_{\lambda\in\Lambda} \lambda\cdot\nabla f,\]
where $W^{1,p}=W^{1,p}(\R^d)$ denotes the $L^p$-Sobolev space of order $1$. In particular, the operator $B\colon W^{2,p}\to L^p,\, f\mapsto \sup_{\lambda\in\Lambda} A_\lambda f$, is well-defined and Lipschitz. Now let $f\in W^{2,p}$, and let $(f_n)_{n\in\mathbb{N}}$ be a sequence in $C_c^\infty$ with $\|f-f_n\|_{W^{2,p}} \to 0$. By the Lipschitz continuity of $B$, it follows that $(Bf_n)_{n\in\mathbb{N}}$ is a Cauchy sequence in $L^p$ and therefore convergent. By Theorem~\ref{thm:lpenvelope}, we have $Af_n = Bf_n$ for all $n\in \N$, and as the generator $A$ of $S$ is closed due to Proposition \ref{prop:genclosed}, we obtain $f\in D(A)$ with $Af=Bf$. Therefore, we see that $B\subset A$ (see Theorem \ref{thm:uniqueness1}). As the nonlinearity $F$ is Lipschitz continuous as a map from $W^{1,p}$ to $L^p$, it can be shown that all assumptions of \cite[Prop. 7.1.10 (iii)]{Lunardi95} are satisfied. Therefore, for every $f\in W^{2,p}$ there exists a solution $u\in C^1([0,\infty); L^p)$ with $u(t)\in W^{2,p}$ for all $t\geq 0$ that solves the Cauchy problem
\[ u'(t) = B u(t) \quad\text{for all }t>0,\quad u(0)=f.\]
By Theorem \ref{thm:uniqueness1}, it follows that $u(t)=S(t)f$ for all $t\ge 0$ and $f\in W^{2,p}$. In particular, $W^{2,p}$ is $S(t)$-invariant for all $t\geq 0$. Therefore, $S$ is the unique continuous extension of the solution operator $f \mapsto u(\cdot,f)$, which is defined on $W^{2,p}$. Moreover, we obtain the existence of a classical solution to $u'=Au$ for initial data in $D(A)$, which is a superset of $W^{2,p}$. Notice that we did not use results from PDE theory in order to obtain the well-posedness (in particular the existence and uniqueness of a solution) of the above Cauchy problem \eqref{eq.cauchygexp}.
\end{example}


\begin{example}[Compound Poisson processes]\label{ex.compoundpois}
Let $\mu\colon \BB(\R^d)\to [0,1]$ be a fixed probability measure. For $\la\geq 0$, $t\geq 0$, $f\in L^p$ and $x\in \R^d$, let
\[
 \big(S_\la(t)f\big)(x):=e^{-\la t}\sum_{n=0}^\infty\frac{(\la t)^n}{n!} \int_{\R^d}\cdots \int_{\R^d} f(x+y_1+\ldots +y_n)\, {\rm d}\mu(y_1)\cdots {\rm d}\mu(y_n).
\]
Then, $S_\la$ is the semigroup corresponding to a compound Poisson process with intensity $\la\geq 0$ and jump distribution $\mu$. Now, let $\La\subset [0,\infty)$ be bounded, $\underline\la:=\inf\La$ and $\overline\la:=\sup\La$. Let
\[
 J_h f:=\sup_{\la\in \La}S_\la(h)f\quad \text{for }h\geq 0\text{ and }f\in L^p.
\]
Then, by Jensen's inequality,
\begin{align*}
\big(J_h f\big)(x)&\leq \bigg(\sup_{\la\in \La}e^{-\la h}\sum_{n=0}^\infty\frac{(\la h)^n}{n!} \int_{\R^d}\cdots \int_{\R^d} |f(x+y_1+\ldots +y_n)|^p\, {\rm d}\mu(y_1)\cdots {\rm d}\mu(y_n)\bigg)^{1/p}\\
&\leq e^{\big(\overline\la -\underline \la\big)h}\big(\big(S_{\overline\la}(h)|f|^p\big)(x)\big)^{1/p}=:\big(C(h)f\big)(x)
\end{align*}
for all $h\geq 0$, $f\in L^p$ and $x\in \R^d$. As before, we see that $C(h_1)C(h_2)=C(h_1+h_2)$ for all $h_1,h_2>0$ and
\[
 J_\pi f\leq C(t_1-t_0)\cdots C(t_m-t_{m-1})f=C(t_m)f
\]
for any partition $\pi=\{t_0,t_1,\dots,t_m\}\in P$ with $0=t_0< t_1< \ldots < t_m$.
Again, by Fubini's theorem,
\[
\|C(h)f\|_p=e^{\big(\overline\la -\underline \la\big) h}\|f\|_p
\]
for all $h\geq 0$ and $f\in L^p$, showing that $C(h)\colon L^p\to L^p$ is bounded. Let $f\in C_c^\infty$. It remains to show that $\frac{1}{h}\int_{\R^d\setminus K}\big|\big(C(h)f\big)(x)\big|^p\, {\rm d}x<\ep$ for $h>0$ sufficiently small. However, this follows from the fact that
\[
 \int_{\R^d}\bigg|\frac{\big(S_{\overline\la}(h)|f|^p\big)(x)-|f(x)|^p}{h}-\overline\la\int_{\R^d}|f(x+y)|^p-|f(x)|^p\, {\rm d}\mu (y)\bigg|\, {\rm d}x\to 0\quad \text{as }h\downarrow 0.
\]
By Theorem \ref{thm:lpenvelope}, the semigroup envelope $S=(S(t))_{t\in [0,\infty)}$ of $(S_\la)_{\la\in \La}$ exists, and is a monotone, bounded and sublinear $C_0$-semigroup. Let $B\colon L^p\to L^p$ be given by
\[
 (Bf)(x):=\sup_{\la\in \La}\la\int_{\R^d}\big(f(x+y)-f(y)\big)\,{\rm d}\mu(y)\quad \text{for }f\in L^p\text{ and }x\in \R^d.
\]
Then, we have $A=B$ on $C_c^\infty$ by Theorem~\ref{thm:lpenvelope}. Since $B$ is bounded and sublinear, and thus globally Lipschitz (see Lemma~\ref{A.6}), $A$ is closed by Proposition \ref{prop:genclosed} and $C_c^\infty$ is dense in $L^p$, it follows that $D(A)=L^p$ and therefore $A=B$. In particular, we obtain a classical solution in the sense of Corollary~\ref{cor:wellposed} to the initial value problem
\[
 u'(t)=Au(t)=Bu(t)\quad \text{for all }t\geq 0, \quad u(0)=f, \]
for all initial values $f\in L^p$.

Finally, we remark that due to the global Lipschitz continuity of $B$, we can also apply  the theorem of Picard-Lindel\"of to obtain a unique solution $u=u(\cdot,f)$ to the abstract initial value problem
\[
 u'(t)=Bu(t)\quad \text{for all }t\geq0, \quad u(0)=f, \]
 for all $f\in L^p$. By Theorem \ref{thm:uniqueness1}, it follows that $u(t,f) = S(t)f$ for all $t\ge 0$ and $f\in L^p$.
\end{example}

\begin{appendix}
	
	\section{Bounded convex operators}
	Let $X$ and $Y$ be Banach lattices.
	For an operator $S\colon X\to Y$, we define $S_x\colon X\to Y$ by $S_x y:=S(x+y)-Sx$ for all $x,y\in X$. Recall that $S\colon X\to Y$ is bounded, if $\|S\|_r<\infty$ for all $r>0$, where
	\[
	\|S\|_r:=\sup_{x\in B(0,r)}\|S x\|.
	\]
	Here, $B(x_0,r):=\{x\in X\colon \|x-x_0\|\leq r\}$ for $x_0\in X$ and  $r>0$.
	
	\begin{lemma}\label{lem:lip}
		Let $S\colon X\to Y$ be convex with $S 0=0$ and $r>0$ with $b:=\|S\|_r<\infty$. Then,
		\[
		\|S x\|\leq \tfrac{2b}{r}\|x\|
		\]
		for all $x\in B(0,r)$.
	\end{lemma}
	
	\begin{proof}
		Let  $x\in B(0,r)$. For $x=0$, the statement holds by assumption. For $x\neq 0$, the convexity of $S$ implies that
		\[
		S x\leq \tfrac{\|x\|}{r} S\Big( \tfrac{r}{\|x\|}x\Big)
		\quad\mbox{and}\quad S x\geq -S(-x) \geq -\tfrac{\|x\|}{r} S\Big( -\tfrac{r}{\|x\|}x\Big),
		\]
		so that,
		\[
		\|S x\|\leq \tfrac{\|x\|}{r} \Big(\big\|S\big( \tfrac{r}{\|x\|}x\big)\big\|+\big\|S\big(- \tfrac{r}{\|x\|}x\big)\big\| \Big)\le \tfrac{2b}{r}\|x\|.
		\]
	\end{proof}
	
	The following two lemmas aim to clarify the difference between convex continuous and convex bounded operators.
	
	\begin{lemma}\label{lemma1}
		Let $S\colon X\to Y$ be convex. Then, the following statements are equivalent:
		\begin{itemize}
			\item[(i)] $S$ is continuous.
			\item[(ii)] For all $x\in X$, there exists some $r>0$ such that $\|S_x\|_r<\infty$.
		\end{itemize}
	\end{lemma}
	
	\begin{proof}
		Let $x\in X$ and $r>0$ with $b:=\|S_x\|_r<\infty$. Then, since $S_x$ is convex with $S_x(0)=0$, we obtain from Lemma \ref{lem:lip} that
		\[
		\|S_x y\|\leq \tfrac{2b}{r} \|y\|\quad\mbox{for all } y\in B(0,r).
		\]
		This shows that $S_x$ is continuous at $0$, i.e., $S$ is continuous at $x$.
		
		Now, assume that there exists some $x\in X$ such that $\|S_x\|_r=\infty$ for all $r>0$. Then, there exists a sequence $(y_n)_n$ in $X$ with $y_n\to 0$ and $\|S_x y_n\|\geq n$. Therefore, the sequence $(S_x y_n)_n$ in $Y$ is unbounded, and thus not convergent. This shows that $S_x$ is not continuous at $0$, i.e., $S$ is not continuous at $x$.
	\end{proof}
	
	\begin{lemma}\label{boundedness}
		Let $S\colon X\to Y$. Then, the following statements are equivalent:
		\begin{itemize}
			\item[(i)] $S$ is bounded.
			\item[(ii)] For all $x\in X$ and all $r>0$, it holds $\|S_x\|_r<\infty$.
		\end{itemize}
	\end{lemma}
	
	\begin{proof}
		Clearly, (ii) implies (i) by considering $x=0$ in (ii). Therefore, assume that $S$ is bounded. Then, for every $x\in X$ and $r>0$, one has $\|S_x\|_r\leq 2\|S\|_{\|x\|+r}<\infty$.
	\end{proof}
\begin{corollary}\label{cor:Lip}
Let $S\colon X\to Y$ be bounded and convex. Then, $S$ is Lipschitz on bounded subsets, i.e., for every $r>0$, there exists some $L>0$ such that $\|Sx-Sy\|\leq L \|x-y\|$ for all $x,y\in B(0,r)$.
\end{corollary}
\begin{proof}
Let $x,y\in B(0,r)$, so that $x-y\in B(0,2r)$. As in the proof of Lemma \ref{boundedness}, it follows that
\[
\|S_x\|_{2r}\leq 2\|S\|_{\|x\|+2r}\leq 2\|S\|_{3r}=:b.
\]
Hence, it follows from Lemma \ref{lem:lip} that $\|Sy-Sx\|=\|S_x(y-x)\|\leq \tfrac{b}{r}\|y-x\|$.
\end{proof}

	In the previous two lemmas, we have seen that, for a convex operator $S\colon X\to Y$, boundedness implies continuity. The following example shows that a convex and continuous operator $S\colon X\to Y$ is not necessarily bounded.
	
	\begin{example}
		Let $X=c_0:=\big\{(x_n) \text{ in }\R\colon x_n\to 0 \text{ as }n\to \infty\big\}$ be endowed with the supremum norm $\|\cdot \|_\infty$ and $Y=\R$. Then, $X$ and $Y$ are two Banach lattices. We define $S\colon X\to Y$ by
		\[
		Sx:=\sup_{n\in \N} |x_n|^n.
		\]
		Notice that $S$ is well-defined, since for every $x\in X$, there exists some $n_0\in \N$ such that $|x_n|\leq 1$ for all $n\in \N$ with $n \geq n_0$. We first show that $S\colon X\to Y$ is convex. For $\la\in [0,1]$ and $x,y\in X$, one has
		\[
		\big|\la x_n+(1-\la )y_n\big|^n\leq \la |x_n|^n+(1-\la )|y_n|^n
		\]
		for all $n\in \N$, which implies that
		\[
		S\big(\la x+(1-\la)y\big)=\sup_{n\in \N}\big|\la x_n+(1-\la )y_n\big|^n\leq \la S x +(1-\la)Sy.
		\]
		Next, we show that $S$ is continuous. Let $x\in X$ and $\ep\in (0,1]$. Then, there exists $n_0\in \N$ such that $|x_n|\leq \tfrac{\ep}{3}$ for all $n\in \N$ with $n\geq n_0$. Now, let $y\in X$ with $\|x-y\|_\infty\le\tfrac{\ep}{3}$ and
		$\|x-y\|_\infty$ is sufficiently small such that
		\[
		\big||x_n|^n-|y_n|^n\big|\leq\ep \quad \text{for all }n\in \N \text{ with }n<n_0.
		\]
		For $n\in \N$ with $n\geq n_0$, one has
		\[
		|x_n|+|y_n|\leq 2|x_n|+\|x-y\|_\infty\leq \ep.
		\]
		Hence, for all $n\in \N$ with $n\geq n_0$,
		\[
		\big||x_n|^n-|y_n|^n\big|\leq |x_n|^n+|y_n|^n\leq |x_n|+|y_n|\leq \ep.
		\]
		Altogether,
		\[
		|Sx-Sy|\leq \sup_{n\in \N} \big||x_n|^n-|y_n|^n\big|\leq \ep.
		\]
		So far, we have shown that $S\colon X\to Y$ is convex and continuous. However, $S$ is not bounded. To that end, let $e_k$ denote the $k$-th unit vector. Then, $2e_k\in B(0,2)$ for all $k\in \N$, but $S(2e_k)=2^k\to \infty$.
	\end{example}
	
	
	In the sublinear case, the notions of continuity and boundedness are equivalent.
	\begin{lemma}\label{lem:sub}
		Let $S\colon X \to Y$ be sublinear. Then, $S$ is bounded, if and only if it is continuous, if and only if it is continuous at 0.
	\end{lemma}
	
	\begin{proof}
		We have already seen that boundedness implies continuity. Therefore, assume that $S$ is continuous at $0$. Then, there exists some $r>0$ such that $\|S\|_r<\infty$. Since $S$ is positive homogeneous, it follows that $\|S\|_r<\infty$ for all $r>0$.
	\end{proof}
	
	\begin{lemma}\label{A.6}
		Let $S\colon X \to Y$ be sublinear and continuous. Then $S$ is Lipschitz, i.e., there exists some $L>0$ such that $\|Sx-Sy\|\leq L \|x-y\|$ for all $x,y\in X$.
	\end{lemma}
	
	\begin{proof}
		Let $L:=2\|S\|_1$ which is finite by Lemma \ref{lem:sub}. Fix $x,y\in X$. By sublinearity, it holds
		\[
		Sx-S y\leq S(x-y)\leq |S(x-y)|+|S(y-x)|.
		\]
		By a symmetry argument, it follows that
		\[
		|Sx-Sy|\leq |S(x-y)|+|S(y-x)|.
		\]
		Hence,
		\[
		\|Sx-Sy\|\leq \|S(x-y)\|+\|S(y-x)\|\leq L\|x-y\|.
		\]
	\end{proof}

	Recall that convex monotone operators are continuous. For the sake of a self-contained exposition, we provide a short proof. We refer, e.g., to B\'atkai et al.~\cite{MR3616245} for a similar proof in the linear case.
	\begin{lemma}\label{continuity}
		Let $S\colon X\to Y$. Then, the following properties hold:
		\begin{itemize}
			\item[a)] If $S$ is convex and monotone, then it is continuous.
			\item[b)] If $S$ is positive homogeneous and monotone, then it is bounded.
		\end{itemize}
	\end{lemma}
	
	\begin{proof}
		a) Let $X_+:=\{x\in X\colon x\geq 0\}$. Suppose by way of contradiction that $S$ is not continuous at $x\in X$. Then there exists a sequence $(x_n)$ in $X$ with $x_n\to 0$ such that
		\[
		\|S_x(x_n)\|=\|S(x_n+x)-S(x)\|\ge \delta
		\]
		for all $n$ and some $\delta>0$.
		Since $S_x$ is monotone it holds $\|S_x x_n\|\le \|S_x|x_n|\|$. Hence, we may assume that $x_n\in X_+$ and $\|x_n\|\le\tfrac{1}{n 2^n}$ for all $n\in \N$ by possibly passing to a subsequence. Define $y:=\sum_{n\in \N} n x_n\in X_+$. Then, for every $\lambda\in (0,1]$ one has
		\[
		\lambda y=\sum_n \lambda n x_n\ge x_n\ge 0
		\]
		for all $n\in \N$ with $\lambda n\ge 1$. By monotonicity of $S$, it follows that $\|S_x(\lambda y)\|\ge \|S_x(x_n)\|\ge\delta$ for all $n\in \N$ large enough. Since $y\in X_+$, the function $[0,1]\to\mathbb{R}$, $\lambda\mapsto \|S_x(\lambda y)\|$ is convex and monotone and therefore continuous at 0. This shows that
		\[
		0=\|S_x(0)\|=\lim_{\lambda\downarrow 0} \|S_x(\lambda y)\|\ge\delta>0,
		\]
		which is a contradiction.
		
		b) Assume that $\|S\|_r=\infty$ for some $r>0$. Then, there exists a sequence $(x_n)$ in $B(0,r)$ with $\|Sx_n\|\geq n 2^n$. As in part a), due to the monotonicity of $S$, we may assume that $x^n\geq 0$. Define $x:=\sum_{n\in \N} 2^{-n} x_n\in B(0,r)$. By monotonicity, we obtain that $0\le S(2^{-n} x_n)\le S(x)$ for all $n\in \N$, so that
		\[\|S x\| \geq \|S(2^{-n} x_n)\|=2^{-n}\|S( x_n)\|\ge n\] for all $n\in \N$. Letting $n\to\infty$, this leads to a contradiction.
	%
	\end{proof}
	
	The results in Section \ref{sec:convexsemigroup} strongly rely on the following uniform boundedness principle for convex continuous operators.
	\begin{theorem}\label{lem:unibound}
		Let $\SS$ be a family of convex continuous operators $X\to Y$. Assume that $\sup_{S\in \SS}\|Sx\|<\infty$ for all $x\in X$.
		\begin{itemize}
			\item[(i)] There exists some $r>0$ such that
			\[
			\sup_{S\in \SS} \|S\|_r<\infty.
			\]
			\item[(ii)] For every $x_0\in X$, there exists some $r>0$ such that
			\[
			\sup_{x\in B(x_0,r)}\sup_{S\in \SS}\|S_x\|_r<\infty.
			\]
		\end{itemize}
	\end{theorem}
	
	\begin{proof}
		(i) By Baire's category theorem, there exist $c>0$, $x_1\in X$ and $r>0$ such that
		\[
		\|S x\|\leq \tfrac{2c}{3}
		\]
		for all $S\in \SS$ and $x\in B(x_1,4r)$. If $x_1=0$, the proof is finished. Hence, assume that $x_1\neq 0$ and define
		\[
		x_0:=\Big(1-\tfrac{2r}{\|x_1\|}\Big)x_1.
		\]
		Since $\|x_0-x_1\|\leq 2r$, it follows that $B(x_0,2r)\subset B(x_1,4r)$. By assumption,
		\[
		d:=\sup_{S\in \SS}\tfrac{1}{2}\|S(-x_0)\|+2\big\|S\big(\tfrac{x_0}{2}\big)\big\|<\infty.
		\]
		Now, let $x\in B(0,r)$ and $S\in \SS$. Then,
		\[
		Sx= S\big(\tfrac{x_0+2x}{2}-\tfrac{x_0}{2}\big) \leq \tfrac{1}{2}\big(S(x_0+2x)+S(-x_0)\big)
		\]
		and
		\[
		2S\big(\tfrac{x_0}{2}\big)-S(x_0-x)=2S\big(\tfrac{x+(x_0-x)}{2}\big)-S(x_0-x)\leq Sx.
		\]
		We thus obtain that
		\begin{align*}
		\|Sx\|&\leq \tfrac{1}{2}\big\|S(x_0+2x)+S(-x_0)\big\|+\big\|2S\big(\tfrac{x_0}{2}\big)-S(x_0-x)\big\|\\
		&\leq \tfrac{1}{2}\|S(x_0+2x)\|+\|S(x_0-x)\|+\tfrac{1}{2}\|S(-x_0)\|+2\big\|S\big(\tfrac{x_0}{2}\big)\big\|\\
		&\leq c+d.
		\end{align*}
		
		(ii) Let $x_0\in X$. Then, $\sup_{S\in \SS}\|S_{x_0}x\|<\infty$ for all $x\in X$.
		By part a), there exist $b\geq 0$ and $r>0$ such that
		\[
		\sup_{S\in \SS}\|S_{x_0}\|_{2r}\leq \tfrac{b}{2}.
		\]
		Now, let $S\in \SS$, $x\in B(x_0,r)$ and $y\in B(0,r)$. Then, $x+y\in B(x_0,2r)$ and
		\[
		S_x y=S_{x_0}(x+y-x_0)-S_{x_0}(x-x_0).
		\]
		Therefore, $\|S_xy\|\leq \|S_{x_0}(x+y-x_0)\|+\|S_{x_0}(x-x_0)\|\leq b$.
	\end{proof}

\end{appendix}


\end{document}